\documentclass[a4paper,12pt]{article}

\usepackage[margin=1in]{geometry}
\usepackage{color}
\usepackage{graphicx}
\usepackage{array}
\usepackage{hyperref}
\usepackage[toc]{appendix}
\usepackage{amsmath}
\usepackage{xcolor}
\usepackage{enumitem}
\usepackage{caption}
\usepackage{subcaption}
\usepackage{float}
\usepackage{amsthm}
\usepackage{mathrsfs}
\usepackage{thmtools}
\usepackage{thm-restate}

\usepackage{cleveref}

\usepackage{amssymb}

\usepackage{tikz}
\usetikzlibrary{arrows}
\newcommand{\midarrowa}{\tikz \draw[-{angle 90}] (0,0)--+(.1,0);}  
\tikzset{every picture/.style={line width=0.9pt}}

\usepackage{algorithmicx}
\usepackage{algpseudocode}

\providecommand{\keywords}[1]
{
  \small	
  \textbf{\textit{Keywords---}} #1
}

\newcommand{\floor}[1]{\left\lfloor #1 \right\rfloor}

\DeclareRobustCommand*{\ora}{\vec}

\newcommand\blfootnote[1]{%
  \begingroup
  \renewcommand\thefootnote{}\footnote{#1}%
  \addtocounter{footnote}{-1}%
  \endgroup
}

	\newtheorem{theorem}{Theorem}
	\newtheorem{corollary}[theorem]{Corollary}
	\newtheorem{lemma}[theorem]{Lemma}
	
	\newtheorem{conjecture}[theorem]{Conjecture}

 	\author{Jasine Babu, Deepu Benson, and Deepak Rajendraprasad \\
        \small Indian Institute of Technology Palakkad \\}
\date{}

\hypersetup{
	colorlinks   = true, 
	urlcolor     = blue, 
	linkcolor    = blue, 
	citecolor    = green 
}

\title{\textbf{Improved Bounds for the Oriented Radius of Mixed Multigraphs}}

\begin{document}

\maketitle
\begin{abstract}

A \emph{mixed multigraph} is a multigraph which may contain both
undirected and directed edges. 
An \emph{orientation} of a mixed multigraph $G$ is an assignment of exactly one direction to each undirected edge of $G$. 
A mixed multigraph $G$ can be oriented to a strongly connected digraph if and only if $G$ is bridgeless and strongly connected [Boesch and Tindell, Am.\ Math.\ Mon., 1980]. 
For each $r \in \mathbb{N}$, let $f(r)$ denote the smallest number such that any strongly connected bridgeless mixed multigraph with radius $r$ can be oriented to a digraph of radius at most $f(r)$. 
We improve the current best upper bound of $4r^2+4r$ on $f(r)$ [Chung, Garey and Tarjan, Networks, 1985] to $1.5 r^2 + r + 1$. 
Our upper bound is tight upto a multiplicative factor of $1.5$ since, $\forall r \in \mathbb{N}$, there exists an undirected bridgeless graph of radius $r$ such that every orientation of it has radius at least $r^2 + r$ [Chv{\'a}tal and Thomassen, J. Comb. Theory. Ser. B., 1978]. 
We prove a marginally better lower bound, $f(r) \geq r^2 + 3r + 1$,
for mixed multigraphs. While this marginal improvement does not help with asymptotic estimates, it clears a natural suspicion that, like undirected graphs,  $f(r)$ may be equal to $r^2 + r$ even for mixed multigraphs.
En route, we show that if each edge of $G$ lies in a cycle of length at most $\eta$, then the oriented radius of $G$ is at most $1.5 r \eta$. 
All our proofs are constructive and lend themselves to polynomial time algorithms.
\end{abstract} 


\blfootnote{Email addresses:
jasine$@$iitpkd.ac.in (Jasine Babu), 
bensondeepu$@$gmail.com (Deepu Benson),
deepak$@$iitpkd.ac.in (Deepak Rajendraprasad)
}

\section{Introduction}
\label{secIntroduction}

A \emph{mixed multigraph} (\emph{mixed graph} for short) is a multigraph in which each edge could be either directed or undirected. A \emph{walk} from a vertex $v_1$ to a vertex $v_k$ in a mixed graph $G$ is a sequence of vertices $v_1, v_2, \ldots, v_k$ such that $\forall i \in [k-1]$, $G$ contains either an undirected edge $\{v_i, v_{i+1}\}$ or an edge directed from $v_i$ to $v_{i+1}$. A \emph{trail} is a walk without repeated edges. A \emph{path} is a trail without repeated vertices, except possibly $v_1 = v_k$. 
A mixed graph is \emph{strongly connected} (\emph{connected} for short) if there is a path from each vertex to every other vertex. 
Notice that this is a common generalisation of connectivity of undirected  graphs and strong connectivity of digraphs.
A \emph{bridge} in a connected mixed graph $G$ is an undirected edge whose removal disconnects $G$.
A bridgeless connected mixed graph is called \emph{$2$-edge connected}.
An \emph{orientation} of a mixed graph $G$ is an assignment of exactly one direction to each undirected edge of $G$. A mixed graph $G$ is called \emph{strongly orientable} if it can be oriented to a strongly connected digraph.

The first major study about strongly orientable graphs was by H.~E.~Robbins in 1939 \cite{robbins1939theorem}, where he proved that an undirected graph is strongly orientable if and only if it is $2$-edge connected. Robbins motivated the study using the real-world scenario of converting a network of two-way streets to one-way streets without sacrificing reachability from any point to another. In the ``real-world'' it is very well possible that some, but not all, of the streets in the network are already one-way. Hence the same scenario motivates the study of strong orientability of mixed graphs.
In 1980, Boesch and Tindell \cite{boesch1980robbins} extended Robbins' theorem to mixed graphs. They proved that a mixed graph is strongly orientable if and only if it is $2$-edge connected. 
Neither of these works approached the problem with a view to contain the distance blow up that may result from the orientation. This was taken up by Chv{\'a}tal and Thomassen in 1978 for undirected graphs \cite{chvatal1978distances}, and Chung, Garey and Tarjan in 1985 for mixed graphs \cite{chung1985strongly}. We need some more terminology before stating their results.

Let $G$ be an undirected, directed or mixed graph. For a vertex $u$ in $G$, $N[u]$ denotes the set consisting of $u$ and all its in-neighbours, out-neighbours and undirected neighbours. The \emph{length} of a path in $G$ is the number of edges (directed and undirected) in the path. The \emph{distance} $d_G(u,v)$ from a vertex $u$ to a vertex $v$ in $G$ is the length of a shortest path from $u$ to $v$ in $G$. The \emph{out-eccentricity} $e_{out}(u)$ of $u$ is $\max\{d_G(u, v) \mid v \in V(G)\}$ and \emph{in-eccentricity} $e_{in}(u)$ is $\max\{d_G(v,u) \mid v \in V(G)\}$. The \emph{eccentricity} $e(u)$ of $u$ is $\max\{e_{out}(u), e_{in}(u)\}$. The \emph{radius} of $G$ is $\min\{e(u) \mid u \in V(G)\}$ and \emph {diameter} is $\max\{e(u) \mid u \in V(G)\}$. A vertex $u$ of $G$ is a central vertex of $G$ if $e(u)$ is equal to the radius of $G$. The radius and diameter of a disconnected graph are taken to be infinite. The \emph{oriented radius (diameter)} of $G$ is the minimum radius (diameter) of an orientation of $G$. 
For each $r \in \mathbb{N}$, let $f(r)$ (resp., $\bar{f}(r)$) denote the smallest number such that any $2$-edge connected mixed graph (resp., undirected graph) with radius $r$ can be oriented to a strongly connected digraph of radius at most $f(r)$ (resp., $\bar{f}(r)$). 
Since every $2$-edge connected undirected graph is also a $2$-edge connected mixed graph, $\bar{f} \leq f$.

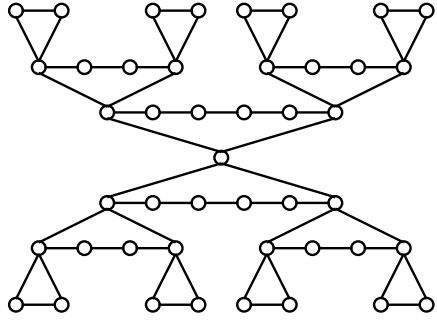
\begin{figure}[tb]
    
    \centering
   
\begin{tikzpicture}[scale=0.15]
\begin{scope}[every node/.style={sloped,allow upside down}]

\draw (0,2) circle (0.6cm);

\draw (2,6) circle (0.6cm);
\draw (6,6) circle (0.6cm);
\draw (10,6) circle (0.6cm);

\draw (-2,6) circle (0.6cm);
\draw (-6,6) circle (0.6cm);
\draw (-10,6) circle (0.6cm);

\draw (12,10) circle (0.6cm);
\draw (16,10) circle (0.6cm);
\draw (8,10) circle (0.6cm);
\draw (4,10) circle (0.6cm);

\draw (-12,10) circle (0.6cm);
\draw (-16,10) circle (0.6cm);
\draw (-8,10) circle (0.6cm);
\draw (-4,10) circle (0.6cm);

\draw (18,15) circle (0.6cm);
\draw (14,15) circle (0.6cm);

\draw (-2,15) circle (0.6cm);
\draw (-6,15) circle (0.6cm);

\draw (2,15) circle (0.6cm);
\draw (6,15) circle (0.6cm);

\draw (-18,15) circle (0.6cm);
\draw (-14,15) circle (0.6cm);

\draw (-17.50,15)--(-14.50,15);
\draw (-5.50,15)--(-2.50,15);
\draw (2.50,15)--(5.50,15);
\draw (14.50,15)--(17.50,15);

\draw (-15.50,10)--(-12.50,10);
\draw (-11.50,10)--(-8.50,10);
\draw (-7.50,10)--(-4.50,10);

\draw (4.50,10)--(7.50,10);
\draw (8.50,10)--(11.50,10);
\draw (12.50,10)--(15.50,10);

\draw (-9.50,6)--(-6.50,6);
\draw (-5.50,6)--(-2.50,6);
\draw (-1.50,6)--(1.50,6);
\draw (2.50,6)--(5.50,6);
\draw (6.50,6)--(9.50,6);

\draw (-16,10.5)--(-18,14.5);
\draw (-4,10.5)--(-6,14.5);
\draw (-16,10.5)--(-14,14.5);
\draw (-4,10.5)--(-2,14.5);

\draw (16,10.5)--(18,14.5);
\draw (4,10.5)--(6,14.5);
\draw (16,10.5)--(14,14.5);
\draw (4,10.5)--(2,14.5);

\draw (10,6.5)--(16,9.5);
\draw (10,6.5)--(4,9.5);

\draw (-10,6.5)--(-16,9.5);
\draw (-10,6.5)--(-4,9.5);

\draw (0,2.5)--(10,5.5);
\draw (0,2.5)--(-10,5.5);


\draw (2,-2) circle (0.6cm);
\draw (6,-2) circle (0.6cm);
\draw (10,-2) circle (0.6cm);

\draw (-2,-2) circle (0.6cm);
\draw (-6,-2) circle (0.6cm);
\draw (-10,-2) circle (0.6cm);

\draw (12,-6) circle (0.6cm);

\draw (16,-6) circle (0.6cm);
\draw (8,-6) circle (0.6cm);

\draw (4,-6) circle (0.6cm);

\draw (-12,-6) circle (0.6cm);

\draw (-16,-6) circle (0.6cm);
\draw (-8,-6) circle (0.6cm);
\draw (-4,-6) circle (0.6cm);

\draw (18,-11) circle (0.6cm);
\draw (14,-11) circle (0.6cm);

\draw (-2,-11) circle (0.6cm);
\draw (-6,-11) circle (0.6cm);

\draw (2,-11) circle (0.6cm);
\draw (6,-11) circle (0.6cm);

\draw (-18,-11) circle (0.6cm);
\draw (-14,-11) circle (0.6cm);

\draw (-17.50,-11)--(-14.50,-11);
\draw (-5.50,-11)--(-2.50,-11);
\draw (2.50,-11)--(5.50,-11);
\draw (14.50,-11)--(17.50,-11);

\draw (-15.50,-6)--(-12.50,-6);
\draw (-11.50,-6)--(-8.50,-6);
\draw (-7.50,-6)--(-4.50,-6);

\draw (4.50,-6)--(7.50,-6);
\draw (8.50,-6)--(11.50,-6);
\draw (12.50,-6)--(15.50,-6);

\draw (-9.50,-2)--(-6.50,-2);
\draw (-5.50,-2)--(-2.50,-2);
\draw (-1.50,-2)--(1.50,-2);
\draw (2.50,-2)--(5.50,-2);
\draw (6.50,-2)--(9.50,-2);

\draw (-16,-6.5)--(-18,-10.5);
\draw (-4,-6.5)--(-6,-10.5);
\draw (-16,-6.5)--(-14,-10.5);
\draw (-4,-6.5)--(-2,-10.5);

\draw (16,-6.5)--(18,-10.5);
\draw (4,-6.5)--(6,-10.5);
\draw (16,-6.5)--(14,-10.5);
\draw (4,-6.5)--(2,-10.5);

\draw (10,-2.5)--(16,-5.5);
\draw (10,-2.5)--(4,-5.5);

\draw (-10,-2.5)--(-16,-5.5);
\draw (-10,-2.5)--(-4,-5.5);

\draw (0,1.5)--(10,-1.5);
\draw (0,1.5)--(-10,-1.5);

\end{scope}
\end{tikzpicture}

    \caption{The undirected graph $G_3$ from \cite{chvatal1978distances}. While $G_3$ has radius $3$, every orientation of $G_3$ has radius at least $12$.}
    \label{fig:fig1}
\end{figure}

Chv{\'a}tal and Thomassen \cite{chvatal1978distances} showed that $\bar{f}(r) = r^2 + r$. 
The upper bound was established by exploiting the properties of a carefully chosen ear-decomposition. 
For the lower bound, they constructed a family of $2$-edge connected graphs $G_r$ of radius $r$ and oriented radius $r^2+r$ (see Figure \ref{fig:fig1} for $G_3$).
Extending the work of Boesch and Tindell on mixed graphs, Chung, Garey and Tarjan \cite{chung1985strongly} in 1985 proposed a linear-time algorithm to check whether a mixed graph has a strong orientation. The resultant orientation could have radius as high as $O(n)$, where $n$ is the number of vertices of the mixed graph. In the same paper, they extend the work of  Chv{\'a}tal and Thomassen on undirected graphs to design a polynomial time algorithm which provides a strong orientation for a $2$-edge connected mixed graph of radius $r$ with oriented radius at most $4r^2 + 4r$. Hence $r^2 + r \leq f(r) \leq 4r^2 + 4r$.

\paragraph{Our results.}
In this article, we establish that $r^2 + 3r - 1 \leq f(r) \leq 1.5r^2 +r + 1$.  We consider the $(8/3)$-factor improvement in the upper bound (over that of Chung, Garey and Tarjan) as the major contribution in this work. The proof is constructive and is presented as an algorithm which is easily seen to be polynomial-time (Section~\ref{secUpperBound}). In comparison, the improvement to the lower bound (over the optimal lower bound for undirected graphs due to Chv{\'a}tal and Thomassen) is marginal and does not help improve asymptotic estimates. But it clears a natural suspicion that $f = \bar{f}$. 
The proof is via a construction of an infinite family of mixed multigraphs parameterised by $r$ (Section~\ref{secLowerBound}). 
Since the diameter $d$ of a mixed graph and its radius $r$ are tied together by $r \leq d \leq 2r$, we can infer that any $2$-edge connected mixed multigraph of diameter $d$ has oriented diameter at most $3d^2 + 2d + 2$. 
In Section~\ref{secSpecialClass}, we illustrate that our algorithm provides a template to handle interesting special cases where there is a bound on the unavoidable penalty you pay for orienting an undirected edge, namely the length of the shortest cycle containing that edge. In particular, we show that if every edge of $G$ lies in a cycle of length at most $\eta$, then our algorithm gives an orientation with radius at most $1.5 r \eta$. 

The gap between our lower and upper bounds on $f(r)$ leaves room for a conjecture on the exact behaviour of $f(r)$. We are inclined to believe that
the upper bound has further room for improvement and show that our algorithm does not perform optimally on the family of graphs constructed to establish our lower bound. We conjecture that $f(r) = r^2 + O(r)$.

\subsubsection*{Further Literature}

Not surprisingly, undirected graphs have overshadowed mixed graphs in this line of research. Moreover, oriented diameter has received more attention than oriented radius. 
For each $d \in \mathbb{N}$, let $\bar{g}(d)$ denote the smallest number such that any $2$-edge connected undirected graph of diameter $d$ has an orientation with diameter at most $\bar{g}(d)$. Chv{\'a}tal and Thomassen \cite{chvatal1978distances} has shown that $\frac{1}{2}d^2+d \leq \bar{g}(d) \leq 2d^2+2d$ for all $d \geq 2$. Since the lower and upper bounds for $\bar{g}$ above differ by a factor of $4$, there have been several attempts to bridge this gap. 
In the same paper, Chv{\'a}tal and Thomassen showed that $\bar{g}(2) = 6$. Kwok, Liu and West \cite{kwok2010oriented} proved that  $9 \leq \bar{g}(3) \leq 11$. Together with Vaka, we showed that $\bar{g}(4) \leq 21$ \cite{10.1007/978-3-030-50026-9_8}. The major effort in the same paper for us was to show that $\bar{g}(d) \leq 1.373d^2 + 6.971d - 1$.
 
Since $\bar{g}$ is quadratic, there have been several studies on special graph classes restricted to which $\bar{g}$ is linear. AT-free graphs \cite{fomin2004free} and chordal graphs \cite{fomin2004complexity} are popular such classes. 
Huang, Li, Li, and Sun \cite{DBLP:journals/dmtcs/HuangLLS14} proved that the oriented radius of an undirected $2$-edge connected graph $G$ is at most $r(\eta-1)$. In this paper we also extend their approach to mixed graphs. 

Bounds (mostly upper) on oriented diameter of $2$-edge connected graphs in terms of other invariants like domination number \cite{fomin2004free}, minimum degree \cite{bau2015diameter}, \cite{surmacs2017improved}, and maximum degree \cite{dankelmann2018oriented} are also known. A final result of Chv{\'a}tal and Thomassen in their 1978 paper was that it is NP-hard to decide whether a given undirected graph has oriented diameter $2$. This has prompted a search for polynomial time algorithms for the problem on special graph classes. Eggemann and Noble \cite{eggemann2009minimizing} studied about the oriented diameter of planar graphs. For each constant $l$, they have proposed an algorithm that decides if a planar graph $G = (V,E)$ has an orientation with diameter at most $l$ and runs in time $O(c|V|)$, where the constant $c$ depends on $l$. 

Orientations of mixed graphs with objectives other than minimising the  radius and diameter have been studied. Aamand, Hjuler, Holm, and Rotenberg \cite{aamand2018one} studied the necessary and sufficient conditions for the existence of a strong orientation for an undirected or a mixed graph when the edges of that graph are partitioned into trails. They also provided a polynomial time algorithm for finding a strong trail orientation, if one exists, for undirected and mixed graphs. Blind \cite{blind2019output} studied $k$-arc-connected orientations in undirected and mixed graphs.  

\section{Upper Bound}
\label{secUpperBound}

In this section, we show that $f(r) \leq 1.5r^2 + r + 1$ by designing and analysing an orientation algorithm {\sc StrongOrientation} that takes a mixed graph $G$ of radius $r$ as input and outputs an orientation of $G$ of radius at most $1.5r^2+r+1$. This algorithm involves repeated applications of two symmetric orientation algorithms, Algorithm {\sc OrientOut} and Algorithm {\sc OrientIn}. To decide the order of execution of these two algorithms, we need to partition the sequence of first $r$ positive odd numbers into two nearly equal parts. The following lemma will help us there.

\begin{restatable}{lemma}{odds}
\label{oddlemma}
For every $n$, the set of first $n$ odd numbers $S=\{2i-1 \mid 1\leq i \leq n \}$ can be partitioned into two sets $A$ and $B$ such that
\begin{equation*}
 \lvert \Sigma_{a \in A} a - \Sigma_{b \in B} b\rvert \leq
    \begin{cases}
      2, & n=2 \\
      1, & n \neq 2
    \end{cases}       
\end{equation*}

Moreover, this partition can be generated in polynomial time.  
\end{restatable}

The proof of Lemma~\ref{oddlemma} is an interesting exercise. For the sake of completeness, we include it in the appendix.

We also use the following lemma during the design of our orientation algorithms. This is a generalization of a lemma by Chung, Garey and Tarjan \cite[Lemma~$5$]{chung1985strongly}. The statement of the original lemma and the proof of Lemma~\ref{tarlemma} are given in the appendix.  

\begin{restatable}{lemma}{goldbach}
\label{tarlemma}
Let $G$ be a strongly connected bridgeless mixed multigraph of radius $r$. Then any edge $pq$ of $G$ is on a cycle of length at most $2r+1$.  
\end{restatable}

Let $G$ be a strongly connected bridgeless mixed graph of radius $r$ with a central vertex $u$. We first describe Algorithm {\sc OrientOut}, in detail, which identifies a subgraph $H$ of $G$ such that $V(H)$ is an $(r-1)$-step dominating set of $V(G)$ and an orientation $\ora{H}$ of $H$ such that $\forall v \in V(\ora{H})$, $d_{\ora{H}}(u,v) \leq 2r$ and $d_{\ora{H}}(v,u) \leq 4r-1$. Using this algorithm alone, we can show that $f(r) \leq 2r^2+r$ using an induction on $r$, as discussed below.  

When $r=0$, the graph $G$ is a single isolated vertex and the assumption holds true trivially. Let us assume the assumption holds true for graphs of radius $r-1$. Let us contract the subgraph $H$ into a single vertex $v_H$ to obtain a graph $G'$. We can see that $G'$ has radius at most $r-1$. Thus by the induction hypothesis, $G'$ has an orientation with radius at most $2(r-1)^2+(r-1)=2r^2-3r+1$. Notice that, $G'$ and $H$ do not share any edges. Thus by combining the orientations of $G'$ and $H$, we can get an orientation $\ora{G}$ of $G$ with radius at most $(2r^2-3r+1) + (4r-1) = 2r^2+r$. Hence, by induction, the assumption is true for all values of $r$.      

The description of the $4$-stage algorithm {\sc OrientOut} is given below.  
 
\subsection*{Algorithm {\sc OrientOut}}

\newcommand{\emsec}[1]{
			\medskip
			\noindent
			\emph{#1}
			}

\emsec{Input:} A strongly connected bridgeless mixed multigraph $G$ and a vertex $u \in V(G)$ with eccentricity at most $r$.

\emsec{Output:} An orientation $\ora{H}$ of a subgraph $H$ of $G$ such that $N[u] \subseteq V(\ora{H}) $ and $\forall v \in V(\ora{H})$, $d_{\ora{H}}(u,v) \leq 2r$ and $d_{\ora{H}}(v,u) \leq 4r-1$.

We create $\ora{H}$ in $4$ stages. The vertices newly added to $\ora{H}$ in a stage will be referred to as the vertices \emph{captured} in that stage.

\emsec{Stage $0$.} Let $v$ be a vertex having multiple edges incident with $u$. If possible, these edges are oriented in such a way that all the $uv$ edges are part of a directed $2$-cycle. Notice that this does not increase any pairwise distances in $G$. We can also remove multiple oriented edges in the same direction between any pair of vertices without affecting any distance. Hence the only multi-edges left in $G$ are those which form a directed $2$-cycle. 

Let $X$ denote the set of all vertices with at least one edge incident with the vertex $u$. $X$ is partitioned into $X_{in}$, $X_{out}$ and $X_{un}$. A vertex $v \in X$ is said to be in $X_{in}$ if it has at least one directed edge towards $u$.  A vertex $v \in X \setminus X_{in}$ is said to be in $X_{out}$ if it has at least one directed edge from $u$. $X_{un} = X \setminus (X_{in} \cup X_{out})$. Notice that a vertex $v \in X_{un}$ has exactly one undirected edge incident with $u$. We initialize $X_{conf} = \emptyset$ (we will later identify this as the set of conflicted vertices in $X$). For each $v \in X$, let $l(v)$ denote the length of a shortest cycle containing an edge between $u$ and $v$. By Lemma~\ref{tarlemma}, $l(v) \leq 2r+1$. Let $s = \Sigma_{v \in X} l(v)$. 

\emsec{Stage $1$.}  Let $G = G_0$. We will orient some of the unoriented edges of $G_0$ incident with $u$ in this stage to obtain a mixed graph $G_1$ as follows.  Let $v_1, v_2, \ldots , v_k$ be an arbitrary ordering of vertices in $X$ such that all the vertices in $X_{in}$ comes before any vertex in $X_{out}$ and all the vertices in $X_{out}$ comes before any vertex in $X_{un}$. Let $c= \lvert X_{in} \cup X_{out} \rvert$. 
\

~\\
Repeat Steps $(i)$ to $(iii)$ for $i=c+1$ to $\lvert X \rvert$:
\begin{enumerate}[label=\roman*.]

\item
An edge $uv_i$ is oriented from $v_i$ to $u$ if the parameter $s$ remains the same even after such an orientation. The vertex $v_i$ is added to $X_{in}$ in this case.

\item
Otherwise, the edge $uv_i$ is oriented from $u$ to $v_i$ if the parameter $s$ remains the same even after such an orientation. The vertex $v_i$ is added to $X_{out}$ in this case.

\item
Otherwise, we leave the edges $uv_i$ unoriented. Such an edge $uv_i$ is called a \emph{conflicted edge} and the vertex $v_i$ is called a \emph{conflicted vertex}. The vertex $v_i$ is added to $X_{conf}$ in this case.  
 
\end{enumerate}   

\emsec{Observation $1$.} If an edge $uv_i$ is conflicted then there exists an edge $\overrightarrow{v_ju}$, at the time of processing $v_i$, where $v_j \in X_{in}$ and $j<i$ such that the edge $uv_i$ is a part of every shortest cycle containing the edge $\overrightarrow{v_ju}$. Otherwise the parameter $s$ would have remained the same even if the edge $uv_i$ gets oriented from $u$ to $v_i$. Hence, for every vertex $v \in X_{conf}$, $uv$ is an undirected edge and there exists a $w \in X_{in}$ such that every shortest path from $u$ to $w$ starts with the edge $uv$.

We have the following distance bounds in $G_1$: $\forall x \in X_{in}$, $d_{G_1}(u,x) \leq 2r$ and $\forall y \in X_{out}$, $d_{G_1}(y,u) \leq 2r$. These bounds follow from the fact that the edges are oriented in Stage $1$ only if the parameter $s$ remains unchanged.

\emsec{Stage~$2$.} In this stage, we try to capture all the vertices in $X_{in}$. Let $T_{out}$ be a minimal tree formed by a breadth first search in $G_1$, starting from $u$ and including at least one path from $u$ to $v$ for each $v \in X_{in}$. Notice that $T_{out}$ has height at most $2r$ and every leaf of $T_{out}$ is a vertex in $X_{in}$. Orient the tree $T_{out}$ outward from $u$ in $G_1$ to obtain $G_2$. Let $S_1 = V(T_{out})$. Every vertex in $S_1$ is part of a directed cycle in $G_2$ containing $u$ and of length at most $2r+1$. Hence for any vertex $x \in S_1$, we have $d_{G_2}(u,x) \leq 2r$ and $d_{G_2}(x,u) \leq 2r$. Notice that, by Observation~$1$, any edge $e=uv$, where $v \in X_{conf}$, will get oriented in Stage $2$. Hence $(X_{in} \cup X_{conf}) \subseteq S_1$.         

\emsec{Stage $3$.} Let $X_{out}' = X_{out} \setminus S_1$. In this stage, we try to capture all the vertices in $X_{out}'$. 
Let $G_3$ be a graph obtained from $G_2$ by contracting $G[S_1]$ into a single vertex $u^*$. 
Every arc $\ora{uv}$, $v \in X_{out}'$, was part of a cycle of length at most $2r+1$ in $G_1$. Hence $v$ has a path of length at most $2r-1$ to some vertex in $S_1$. This path together with the edge $\ora{u^*v}$ ensures that every vertex $v$ in $X_{out}'$ lies in a cycle of length at most $2r$ containing $u^*$ in $G_3$.
Let $T_{in}$ be a minimal tree formed by a reverse breadth first search in $G_3$ starting from $u^*$ such that it contains a path from each $v \in X_{out}'$ to $u^*$. 
Notice that, $T_{in}$ has height at most $2r-1$ and every leaf of $T_{in}$ is a vertex in $X_{out}'$. 
The reduction of $1$ in height is because $(X_{in} \cup X_{conf}) \subseteq S_1$.
Orient the tree $T_{in}$ inward towards $u^*$ in $G_3$ to obtain $G_4$. Let $S_2 = V(T_{in}) \setminus \{u^*\}$. Then for any vertex $y \in S_2$, we have the following bounds, $d_{G_4}(u^*,y) \leq 2r-1$ and $d_{G_4}(y,u^*) \leq 2r-1$. Let $H = G[S_1 \cup S_2]$. Notice that, $T_{out}$ and $T_{in}$ do not share any edges in $H$. The edges of $H$ are oriented consistently with the orientation of $T_{out}$ and $T_{in}$ to obtain $\ora{H}$. We have already proved that for any vertex $x \in S_1$, $d_{G_2}(u,x) \leq 2r$ and $d_{G_2}(x,u) \leq 2r$. Hence $\forall x \in S_1$, $d_{\ora{H}}(u,x) \leq 2r$ and $d_{\ora{H}}(x,u) \leq 2r$. Consider a vertex $y \in S_2$. We have proved that $d_{G_4}(y,u^*) \leq 2r-1$. This implies $d_{\ora{H}}(y,u) \leq (2r-1)+(2r) = 4r-1$. Since height of $T_{in}$ is at most $2r-1$, $y$ has a directed path of length at most $2r-2$ from a vertex $v \in X_{out}'$ in $G_4$. This path along with the arc $\ora{uv}$ will give us the additional bound $d_{\ora{H}}(u,y) \leq 2r$ for all $y \in S_2$. Thus $\ora{H}$ has radius at most $4r-1$ with $d_{\ora{H}}(u,x) \leq 2r$ and $d_{\ora{H}}(x,u) \leq 4r-1$, $\forall x \in V(\ora{H})$.\\ 

As mentioned earlier, using algorithm {\sc OrientOut} and an induction on $r$, we can get the upper bound $f(r) \leq 2r^2+r$. This itself is a $2$-factor improvement over the existing bound of $4r^2+4r$ obtained by Chung, Garey and Tarjan \cite{chung1985strongly}. Notice that, in $\ora{H}$, the bound on $e_{out}(u)$ is better than that of $e_{in}(u)$ by a factor of $2$. To exploit this gap, we design a partner algorithm {\sc OrientIn} using a symmetric construction. 

\subsection*{Algorithm {\sc OrientIn}}

\emsec{Input:} A strongly connected bridgeless mixed multigraph $G$ and a vertex $u \in V(G)$ with eccentricity at most $r$.

\emsec{Output:} An orientation $\ora{H}$ of a subgraph $H$ of $G$ such that $N[u] \subseteq V(\ora{H})$ and $\forall v \in V(\ora{H})$, $d_{\ora{H}}(u,v) \leq 4r-1$ and $d_{\ora{H}}(v,u) \leq 2r$.

Let $G$ be a strongly connected bridgeless mixed graph of radius $r$ and $u$ be a central vertex of $G$. To show that $f(r) \leq 1.5r^2+r+1$, we design an $r$-phase algorithm in which each phase involves an application of either {\sc OrientOut} or {\sc OrientIn}. For convenience, we number the phases from $r$ down to $1$. A graph $G_i$ and a vertex $u_i \in V(G_i)$ with eccentricity at most $i$ are the inputs to the algorithm {\sc OrientOut} or {\sc OrientIn} executed in Phase~$i$ of algorithm {\sc StrongOrientation}. $\ora{H_i}$ is the output of the algorithm executed in Phase~$i$.

\subsection*{Algorithm {\sc StrongOrientation}}

\emsec{Input:} A strongly connected bridgeless mixed multigraph $G$ and a vertex $u \in V(G)$ of eccentricity $r$.

\emsec{Output:} An orientation $\ora{G}$, of $G$, with radius at most $1.5r^2+r+1$.

\begin{algorithmic}[1]
\State $S \gets \{2i-1 \mid 1\leq i \leq r \}$. 
\State Obtain a partition $(A,B)$ of $S$ such that $|\Sigma_{a \in A} a - \Sigma_{b \in B} b| \leq 2$.
\State $I_A \gets \{i \mid (2i-1) \in A \}$, $I_B \gets \{i \mid (2i-1) \in B \}$.   
\State$i \gets r$, $G_i \gets G$, $u_i \gets u$.   
\Repeat
\If {$i \in I_A$} 
        \State $\ora{H_i} \gets ${\sc OrientOut}$(G_i,u_i)$.
\Else        
        \State $\ora{H_i} \gets ${\sc OrientIn}$(G_i,u_i)$.
\EndIf
\State  Obtain $G_{i-1}$ by contracting $V(\ora{H_i})$ to a single vertex $u_{i-1}$. 
\State $i \gets i - 1$.
\Until{$i = 0$}.
\State Take the union of the orientations $\{\ora{H_1}, \ldots , \ora{H_r} \}$ and orient the remaining edges arbitrarily to obtain an orientation $\ora{G}$ of $G$.   
\end{algorithmic}

The existence of the partition used in Step~$2$ of the algorithm is guaranteed by Lemma~\ref{oddlemma} and can be computed in polynomial time.   
Since $N[u_i] \subseteq V(\ora{H_i})$, after each phase, the radius of the resultant graph reduces by at least $1$. Hence, the vertex $u_{i-1}$ of $G_{i-1}$ has eccentricity at most $i-1$. Further notice that for $1 \leq i \leq r$, the graphs $\ora{H_i}$ are pair-wise edge disjoint and hence their union gives a consistent orientation $\ora{G}$ of $G$. For $1 \leq i \leq r$, let $e_{in}^i= \max\{d_{\ora{H_{i}}}(v,u_i) \mid v \in V(\ora{H_i})\}$ denote the in-eccentricity of $u_i$ in $\ora{H_i}$ and $e_{out}^i = \max\{d_{\ora{H_{i}}}(u_i,v) \mid v \in V(\ora{H_i})\}$ denote the out-eccentricity of $u_i$ in $\ora{H_i}$.
By combining the guarantees given by the two algorithms {\sc OrientOut} and {\sc OrientIn}, for the in-distance and out-distance in $\ora{H_i}$, we have the following bounds for $e_{out}(u)$ and $e_{in}(u)$ of $\ora{G}$.     
\begin{itemize}
\item

$ e_{out}(u) \leq \Sigma_{i=1}^{r}e_{out}^i = \Sigma_{i \in I_A} (2i) + \Sigma_{i \in I_B} (4i-1) = \Sigma_{i=1}^r (2i) + \Sigma_{i \in I_B} (2i-1) $. 
\item

$ e_{in}(u) \leq \Sigma_{i=1}^{r}e_{in}^i = \Sigma_{i \in I_A} (4i-1) + \Sigma_{i \in I_B} (2i) = 
\Sigma_{i=1}^r (2i) + \Sigma_{i \in I_A} (2i-1) $.
\end{itemize}

The radius of $\ora{G}$ is at most $\max\{e_{out}(u), e_{in}(u)\} \leq \Sigma_{i=1}^r (2i) + \max\{\Sigma_ {i \in I_A} (2i-1), \Sigma_{i \in I_B} (2i-1) \}$. Notice that $\Sigma_{i=1}^r (2i-1) = r^2$ and by Lemma~\ref{oddlemma}, we get $\max\{\Sigma_ {i \in I_A} (2i-1), \Sigma_{i \in I_B} (2i-1) \}$ is at most $\frac{r^2}{2}+1$. Thus the radius of $\ora{G}$ is at most $\Sigma_{i=1}^r (2i) +  (\frac{r^2}{2}+1) \leq (r^2 + r) + (\frac{r^2}{2}+1) = 1.5r^2+r+1$. Hence, the oriented radius of $G$ is at most $1.5r^2+r+1$.  

This proves the main theorem of the section.  

\begin{theorem}\label{ubmmg}
$f(r) \leq 1.5r^2+r+1$. 
\end{theorem}  

It is easy to verify that Algorithm {\sc StrongOrientation} runs in polynomial time. Hence we have a polynomial-time algorithm to find a strong orientation $\ora{G_r}$ with radius at most $1.5 r^2+r+1$ for a mixed graph $G_r$ of radius $r$. Since $r \leq d \leq 2r$, we can immediately see that a mixed graph of diameter $d$ has an orientation of diameter at most $3d^2+2d^2+2$ for all $d$.  

\section{Upper Bound for Some Special Classes of Mixed Multigraphs}
\label{secSpecialClass}

Let $G$ be a strongly connected mixed graph with radius $r$. Let $\eta(G)$ denote the smallest integer $k$ such that every edge of $G$ belongs to a cycle of length at most $k$. Huang, Li, Li, and Sun \cite{DBLP:journals/dmtcs/HuangLLS14} proved that the oriented radius of an undirected graph $G$ of radius $r$ is at most $r(\eta(G)-1)$. We generalize this result to mixed graphs. We start by generalizing Lemma~$\ref{oddlemma}$ as follows. 

\begin{restatable}{lemma}{oddevens}
\label{oelemma}
Let set $S_1=\{2i-1 \mid 1\leq i \leq n \}$. Let $S_2$ be a multiset with $k$ elements of the same value: either $2n-1$ or $2n$. Then the multiset $S = S_1 \cup S_2$ can be partitioned into two multisets $A$ and $B$ such that $\lvert \Sigma_{a \in A} a - \Sigma_{b \in B} b \rvert \leq 2$. Moreover, this partition can be generated in polynomial time.    
\end{restatable}
The proof of this lemma is simple and is given in the appendix. 

Let $G$ be a strongly connected bridgeless mixed graph with a vertex $u \in V(G)$ having eccentricity at most $r$ and $\eta(G) = \eta$. 
	By Lemma~\ref{tarlemma}, $\eta \leq 2r+1$. 
	We use the same terminology of $G_i$, $u_i$, and $\ora{H_i}$ as used in Algorithm {\sc StrongOrientation}. 
	Let $\eta_{u_i}$ denote the smallest integer $k$ such that every edge incident with $u_i$ belongs to a cycle of length at most $k$ in $G_i$. 	
	Since $u_i$ has eccentricity at most $i$, by Lemma~\ref{tarlemma}, it is easy to see that, $\eta_{u_i} \leq \min\{\eta, 2i+1\}$. Let $ k_i = \min\{\eta, 2i+1\} $ denote this upper bound.   
	From the details of algorithms {\sc OrientIn} and {\sc OrientOut}, one can see that when $G_i$ and $u_i$ are given as input, the distance guarantees of these algorithms only use the value of $\eta_{u_i} \leq k_i$. 
	Now, a second look at algorithm {\sc OrientIn}, with the definition of $\eta_{u_i}$ in mind, will show that the oriented subgraph $\ora{H_i}$ returned by {\sc OrientIn}$(G_i,u_i)$ has $e_{in}(u_i)$ at most $(k_{i}-1)$ and $e_{out}(u_i)$ at most $(2k_{i}-3)$. 
	Similarly, $\ora{H_i}$ returned by {\sc OrientOut}$(G_i,u_i)$ has $e_{out}(u_i)$ at most $(k_{i}-1)$ and $e_{in}(u_i)$ at most $(2k_{i}-3)$.

A strong orientation $\ora{G}$ of $G$ is obtained by an $r$-phase algorithm similar to {\sc StrongOrientation}, with each phase involving an execution of either {\sc OrientIn} or {\sc OrientOut}. Let $m = \floor{(\eta -1)/2}$. Recall that $2i+1$ dominates $\eta$ for $i > m$. For $1 \leq i \leq r$, let $e_{in}^i= \max\{d_{\ora{H_{i}}}(vu_i) \mid v \in V(\ora{H_i})\}$ denote the in-eccentricity of $u_i$ in $\ora{H_i}$ and $e_{out}^i = \max\{d_{\ora{H_{i}}}(u_iv) \mid v \in V(\ora{H_i})\}$ denote the out-eccentricity of $u_i$ in $\ora{H_i}$.
Let $S$ be the multiset $ \{k_i-2  \mid 1 \leq i \leq r \}$. It can be verified that $S$ satisfies the requirements of Lemma~\ref{oelemma}. Hence, we can obtain a partition $(A,B)$ of $S$ such that their sums differ at most by $2$.
 Let $\sigma = (\sigma_1,  \ldots , \sigma_r)$ be a non-decreasing ordering of elements in $S$. For $1\leq i \leq r $, let $i \in I_A$ if and only if $\sigma_i \in A$ and $i \in I_B$ otherwise. In Phase~$i$, we apply  {\sc OrientOut} if $i \in I_A$ and {\sc OrientIn} if $i \in I_B$. By combining the guarantees given by {\sc OrientOut} and {\sc OrientIn} for the in-distance and out-distance in $\ora{H_i}$, we can have the following bounds for $e_{in}(u)$ and $e_{out}(u)$ of $\ora{G}$.     

\begin{align*}
e_{in}(u) & \leq \Sigma_{i=1}^{r}e_{in}^i \\ & \leq \Sigma_{i \in I_A} (2k_i-3) + \Sigma_{i \in I_B}(k_i-1) \\ & =
\Sigma_{i=1}^r (k_i-1) + \Sigma_{i \in I_A} (k_i-2) \\ & =:k_{in}
\end{align*}
\begin{align*}
e_{out}(u)  & \leq \Sigma_{i=1}^{r}e_{out}^i \\ & \leq \Sigma_{i \in I_A} (k_i-1) + \Sigma_{i \in I_B} (2k_i-3) \\ &=
\Sigma_{i=1}^r (k_i-1) + \Sigma_{i \in I_B} (k_i-2) \\ & =: k_{out}
\end{align*}
Further, we have the following bound for $\lvert k_{in} - k_{out} \rvert$ and $(k_{in} + k_{out})$.

\begin{align}
 \lvert k_{in} - k_{out} \rvert &=
\lvert \Sigma_{i \in I_A} (k_i-2) - \Sigma_{i \in I_B} (k_i-2) \rvert \nonumber \\ & \leq 2 ~~~~~(\text{Lemma}~\ref{oelemma}) 
\end{align} 
\begin{align}
(k_{in} + k_{out})  & = 
2 \Sigma_{i=1}^r (k_i-1) + \Sigma_{i=1}^r (k_i-2) \nonumber \\&=
(3 \Sigma_{i=1}^r k_i)-4r \nonumber \\&=
3\left[ (\Sigma_{i=1}^{m} 2i+1) + (\Sigma_{i=m+1}^{r} \eta ) \right] -4r \nonumber ~~~~~(\text{By definition of }k_i) \\&= 
3\left[ (\Sigma_{i=1}^{r} \eta) - \Sigma_{i=1}^{m} (\eta-(2i+1)) \right] -4r \nonumber \\&= 3r \eta - 3\Sigma_{i=1}^{m} (\eta-(2i+1)) -4r
\end{align} 
 
Notice that, $\eta \in \{2m+1,2m+2\}$. Now consider the term $\Sigma_{i=1}^{m} (\eta-(2i+1))$ in Equation~$2$. Notice that, when $\eta$ is even, this term sums up odd numbers from $1$ to $\eta -3$ and when $\eta$ is odd, the term sums up even numbers from $0$ to $\eta -3$. 
Hence, by taking the worst case for Equation~$2$, we get.

\begin{align}
(k_{in} + k_{out})  \leq 
3 r \eta - \frac{3}{4}(\eta-1)(\eta-3) -4r 
\end{align} 

From equations $1$ and $3$, we get.
\begin{align}
\max\{e_{in}(u), e_{out}(u)\} &\leq \max\{k_{in}, k_{out}\} \nonumber \\ & \leq \floor{\frac{k_{in}+k_{out}}{2}}+1 \nonumber \\ & \leq 1.5 r \eta - 0.375(\eta-1)(\eta-3)-2r +1
\end{align}

Since $\eta \geq 3$, $(\eta-1)(\eta-3)$ is non-negative. Hence, by Equation $4$, we have the following theorem.

\begin{theorem}\label{etatheorem}
Let $G$ be a strongly connected bridgeless mixed multigraph of radius $r$ and let $\eta(G)=\eta$ denote the smallest integer $k$ such that every edge of $G$ belongs to a cycle of length at most $k$. Then $G$ has oriented radius at most $1.5 r \eta - 0.375(\eta-1)(\eta-3)-2r +1 \leq 1.5 r \eta $.  
\end{theorem}  

A $k$-chordal mixed multigraph is a mixed multigraph graph in which all cycles of length more than $k$ have a chord. In particular when $k = 3$, the graph is called a chordal mixed multigraph. Since $\eta \leq k$, by Theorem~\ref{etatheorem}, we have the following corollary.   

\begin{corollary}\label{kcorollary}  
Let $G$ be a strongly connected bridgeless $k$-chordal mixed multigraph with radius $r$. Then $G$ has oriented radius at most $1.5 r k - 0.375(k-1)(k-3)-2r +1$. In particular, if $G$ is a strongly connected bridgeless chordal mixed multigraph of radius $r$ then $G$ has oriented radius at most $2.5r+1$.  
\end{corollary}


\section{Lower Bound}
\label{secLowerBound}

\begin{figure}[h]
    
    \centering
   
\begin{tikzpicture}[scale=0.28]
\begin{scope}[every node/.style={sloped,allow upside down}]
\draw (0,2) circle (0.3cm); 

\draw (0,6) circle (0.3cm); 

\draw (2,6) circle (0.3cm);
\draw (4,6) circle (0.3cm);
\draw (6,6) circle (0.3cm);
\draw (8,6) circle (0.3cm);
\draw (10,6) circle (0.3cm); 

\node [label={[xshift=0.5cm, yshift=0.18cm]$a$}]{};
\node [label={[xshift=-0.044cm, yshift=1.6cm]$c$}]{};
\node [label={[xshift=-3.2cm, yshift=1.2cm]$b$}]{};
\node [label={[xshift=3.2cm, yshift=1.2cm]$d$}]{};
\node [label={[xshift=-5.4cm, yshift=2.68cm]$f$}]{};
\node [label={[xshift=5.4cm, yshift=2.68cm]$g$}]{};

\draw (-2,6) circle (0.3cm);
\draw (-4,6) circle (0.3cm);
\draw (-6,6) circle (0.3cm);
\draw (-8,6) circle (0.3cm);
\draw (-10,6) circle (0.3cm); 

\draw (10,9) circle (0.3cm);

\draw (12,9) circle (0.3cm);
\draw (14,9) circle (0.3cm);
\draw (16,9) circle (0.3cm);
\draw (8,9) circle (0.3cm);
\draw (6,9) circle (0.3cm);
\draw (4,9) circle (0.3cm);

\draw (-10,9) circle (0.3cm);

\draw (-12,9) circle (0.3cm);
\draw (-14,9) circle (0.3cm);
\draw (-16,9) circle (0.3cm);
\draw (-8,9) circle (0.3cm);
\draw (-6,9) circle (0.3cm);
\draw (-4,9) circle (0.3cm);

\draw (16,11.5) circle (0.3cm);

\draw (18,11.5) circle (0.3cm);
\draw (14,11.5) circle (0.3cm);

\draw (-4,11.5) circle (0.3cm);

\draw (-2,11.5) circle (0.3cm);
\draw (-6,11.5) circle (0.3cm);

\draw (4,11.5) circle (0.3cm);

\draw (2,11.5) circle (0.3cm);
\draw (6,11.5) circle (0.3cm);

\draw (-16,11.5) circle (0.3cm);

\draw (-18,11.5) circle (0.3cm);
\draw (-14,11.5) circle (0.3cm);

\draw (-17.80,11.5)--(-16.20,11.5);
\draw (-15.80,11.5)--(-14.20,11.5);

\draw (-5.80,11.5)--(-4.20,11.5);
\draw (-3.80,11.5)--(-2.20,11.5);

\draw (2.20,11.5)--(3.80,11.5);
\draw (4.20,11.5)--(5.80,11.5);

\draw (14.20,11.5)--(15.80,11.5);
\draw (16.20,11.5)--(17.80,11.5);

\draw (-17.20,11.5)--node {\midarrowa}(-16.20,11.5);
\draw (-15.20,11.5)--node {\midarrowa}(-14.20,11.5);

\draw (-5.20,11.5)--node {\midarrowa}(-4.20,11.5);
\draw (-3.20,11.5)--node {\midarrowa}(-2.20,11.5);

\draw (2.80,11.5)--node {\midarrowa}(3.80,11.5);
\draw (4.80,11.5)--node {\midarrowa}(5.80,11.5);

\draw (14.80,11.5)--node {\midarrowa}(15.80,11.5);
\draw (16.80,11.5)--node {\midarrowa}(17.80,11.5);

\draw (-15.80,9)--(-14.20,9);
\draw (-13.80,9)--(-12.20,9);
\draw (-13.20,9)--node {\midarrowa}(-12.20,9);
\draw (-11.80,9)--(-10.20,9);
\draw (-9.80,9)--(-8.20,9);
\draw (-7.80,9)--(-6.20,9);
\draw (-7.20,9)--node {\midarrowa}(-6.20,9);
\draw (-5.80,9)--(-4.20,9);

\draw (4.20,9)--(5.80,9);
\draw (6.20,9)--(7.80,9);
\draw (6.80,9)--node {\midarrowa}(7.80,9);
\draw (8.20,9)--(9.80,9);
\draw (10.20,9)--(11.80,9);
\draw (12.20,9)--(13.80,9);
\draw (12.80,9)--node {\midarrowa}(13.80,9);
\draw (14.20,9)--(15.80,9);

\draw (-9.80,6)--(-8.20,6);
\draw (-7.80,6)--(-6.20,6);
\draw (-5.80,6)--(-4.20,6);
\draw (-5.20,6)--node {\midarrowa}(-4.20,6);
\draw (-3.80,6)--(-2.20,6);
\draw (-1.80,6)--(-0.20,6);
\draw (0.20,6)--(1.80,6);
\draw (2.20,6)--(3.80,6);
\draw (4.20,6)--(5.80,6);
\draw (4.80,6)--node {\midarrowa}(5.80,6);
\draw (6.20,6)--(7.80,6);
\draw (8.20,6)--(9.80,6);

\draw (0,2.2)--(0,5.8);

\draw (10,6.2)--(10,8.8);
\draw (-10,6.2)--(-10,8.8);

\draw (16,9.2)--(16,11.3);
\draw (-16,9.2)--(-16,11.3);
\draw (4,9.2)--(4,11.3);
\draw (-4,9.2)--(-4,11.3);

\draw (-16,9.2)--(-18,11.3);
\draw (-4,9.2)--(-6,11.3);
\draw (-16,9.2)--(-14,11.3);
\draw (-4,9.2)--(-2,11.3);

\draw (16,9.2)--(18,11.3);
\draw (4,9.2)--(6,11.3);
\draw (16,9.2)--(14,11.3);
\draw (4,9.2)--(2,11.3);

\draw (10,6.2)--(16,8.8);
\draw (10,6.2)--(4,8.8);

\draw (-10,6.2)--(-16,8.8);
\draw (-10,6.2)--(-4,8.8);

\draw (0,2.2)--(10,5.8);
\draw (0,2.2)--(-10,5.8);


\draw (0,-2) circle (0.3cm);

\draw (2,-2) circle (0.3cm);
\draw (4,-2) circle (0.3cm);
\draw (6,-2) circle (0.3cm);
\draw (8,-2) circle (0.3cm);
\draw (10,-2) circle (0.3cm);

\draw (-2,-2) circle (0.3cm);
\draw (-4,-2) circle (0.3cm);
\draw (-6,-2) circle (0.3cm);
\draw (-8,-2) circle (0.3cm);
\draw (-10,-2) circle (0.3cm);

\draw (10,-5) circle (0.3cm);

\draw (12,-5) circle (0.3cm);
\draw (14,-5) circle (0.3cm);
\draw (16,-5) circle (0.3cm);
\draw (8,-5) circle (0.3cm);
\draw (6,-5) circle (0.3cm);
\draw (4,-5) circle (0.3cm);

\draw (-10,-5) circle (0.3cm);

\draw (-12,-5) circle (0.3cm);
\draw (-14,-5) circle (0.3cm);
\draw (-16,-5) circle (0.3cm);
\draw (-8,-5) circle (0.3cm);
\draw (-6,-5) circle (0.3cm);
\draw (-4,-5) circle (0.3cm);

\draw (16,-7.5) circle (0.3cm);

\draw (18,-7.5) circle (0.3cm);
\draw (14,-7.5) circle (0.3cm);

\draw (-4,-7.5) circle (0.3cm);

\draw (-2,-7.5) circle (0.3cm);
\draw (-6,-7.5) circle (0.3cm);

\draw (4,-7.5) circle (0.3cm);

\draw (2,-7.5) circle (0.3cm);
\draw (6,-7.5) circle (0.3cm);

\draw (-16,-7.5) circle (0.3cm);

\draw (-18,-7.5) circle (0.3cm);
\draw (-14,-7.5) circle (0.3cm);

\draw (-17.80,-7.5)--(-16.20,-7.5);
\draw (-15.80,-7.5)--(-14.20,-7.5);

\draw (-5.80,-7.5)--(-4.20,-7.5);
\draw (-3.80,-7.5)--(-2.20,-7.5);

\draw (2.20,-7.5)--(3.80,-7.5);
\draw (4.20,-7.5)--(5.80,-7.5);

\draw (14.20,-7.5)--(15.80,-7.5);
\draw (16.20,-7.5)--(17.80,-7.5);

\draw (-17.20,-7.5)--node {\midarrowa}(-16.20,-7.5);
\draw (-15.20,-7.5)--node {\midarrowa}(-14.20,-7.5);

\draw (-5.20,-7.5)--node {\midarrowa}(-4.20,-7.5);
\draw (-3.20,-7.5)--node {\midarrowa}(-2.20,-7.5);

\draw (2.80,-7.5)--node {\midarrowa}(3.80,-7.5);
\draw (4.80,-7.5)--node {\midarrowa}(5.80,-7.5);

\draw (14.80,-7.5)--node {\midarrowa}(15.80,-7.5);
\draw (16.80,-7.5)--node {\midarrowa}(17.80,-7.5);

\draw (-15.80,-5)--(-14.20,-5);
\draw (-13.80,-5)--(-12.20,-5);
\draw (-13.20,-5)--node {\midarrowa}(-12.20,-5);
\draw (-11.80,-5)--(-10.20,-5);
\draw (-9.80,-5)--(-8.20,-5);
\draw (-7.80,-5)--(-6.20,-5);
\draw (-7.20,-5)--node {\midarrowa}(-6.20,-5);
\draw (-5.80,-5)--(-4.20,-5);

\draw (4.20,-5)--(5.80,-5);
\draw (6.20,-5)--(7.80,-5);
\draw (6.80,-5)--node {\midarrowa}(7.80,-5);
\draw (8.20,-5)--(9.80,-5);
\draw (10.20,-5)--(11.80,-5);
\draw (12.20,-5)--(13.80,-5);
\draw (12.80,-5)--node {\midarrowa}(13.80,-5);
\draw (14.20,-5)--(15.80,-5);

\draw (-9.80,-2)--(-8.20,-2);
\draw (-7.80,-2)--(-6.20,-2);
\draw (-5.80,-2)--(-4.20,-2);
\draw (-5.20,-2)--node {\midarrowa}(-4.20,-2);
\draw (-3.80,-2)--(-2.20,-2);
\draw (-1.80,-2)--(-0.20,-2);
\draw (0.20,-2)--(1.80,-2);
\draw (2.20,-2)--(3.80,-2);
\draw (4.20,-2)--(5.80,-2);
\draw (4.80,-2)--node {\midarrowa}(5.80,-2);
\draw (6.20,-2)--(7.80,-2);
\draw (8.20,-2)--(9.80,-2);

\draw (0,1.8)--(0,-1.8);

\draw (10,-2.2)--(10,-4.8);
\draw (-10,-2.2)--(-10,-4.8);

\draw (16,-5.2)--(16,-7.3);
\draw (-16,-5.2)--(-16,-7.3);
\draw (4,-5.2)--(4,-7.3);
\draw (-4,-5.2)--(-4,-7.3);

\draw (-16,-5.2)--(-18,-7.3);
\draw (-4,-5.2)--(-6,-7.3);
\draw (-16,-5.2)--(-14,-7.3);
\draw (-4,-5.2)--(-2,-7.3);

\draw (16,-5.2)--(18,-7.3);
\draw (4,-5.2)--(6,-7.3);
\draw (16,-5.2)--(14,-7.3);
\draw (4,-5.2)--(2,-7.3);

\draw (10,-2.2)--(16,-4.8);
\draw (10,-2.2)--(4,-4.8);

\draw (-10,-2.2)--(-16,-4.8);
\draw (-10,-2.2)--(-4,-4.8);

\draw (0,1.8)--(10,-1.8);
\draw (0,1.8)--(-10,-1.8);

\end{scope}
\end{tikzpicture}

    \caption{Mixed multigraph $G_3$ of radius $3$ and oriented radius $17$.}
    \label{fig-mmg}
\end{figure}

When we restrict to undirected graphs, the currently known lower bound for $f(r)$ is $r^2+r$. This follows from a construction by Chv{\'a}tal and Thomassen \cite{chvatal1978distances}, an instance of which is shown in Figure~\ref{fig:fig1}. Using a similar idea, we construct a mixed graph $G_r$ of radius $r$ whose oriented radius is $r^2 + 3r - 1$. In order to construct $G_r$, we first define a ternary tree $T_r$ of height $r$, rooted at a vertex $a$. Each internal vertex $v$ of $T_r$ has $3$ children, called the \emph{left}, \emph{middle}, and  \emph{right} child of $v$. The \emph{Level} of a vertex is its distance from $a$ in $T_r$. For a vertex $v$ in Level $i$, for $0 \leq i \leq r-2$, the middle child of $v$ is a leaf in $T_r$ and the other two children are internal vertices. Let $b$, $c$, and $d$ be the left, middle, and right children, respectively, of $a$. Let $f$ be the leaf in Level~$r$ of $T_r$ such that every vertex in the $a-f$ path other than $a$ is a left child of its parent. Let $g$ be the leaf in Level~$r$ of $T_r$ such that every vertex in the $a-g$ path other than $a$ is a right child of its parent. We construct a mixed graph $H$ from $T_r$ as follows. For each Level~$i$, for $i=0$ to $r-1$ and for each internal vertex $p_i$ in Level $i$, with $l_{i+1}$, $m_{i+1}$, $r_{i+1}$, respectively, being the left, right, and middle children of $p_i$, we do the following. Add paths of length $2r-2i-1$ from $l_{i+1}$ to $m_{i+1}$ and  $m_{i+1}$ to $r_{i+1}$. The middle edge of the first path is oriented along the direction from $l_{i+1}$ to $m_{i+1}$ and the middle edge of the second path is oriented along the direction from $m_{i+1}$ to $r_{i+1}$. Notice that, due to the presence of these directed edges, in any strong orientation of $H$, directions of all the edges except the edges between the middle children and their parents are forced. But we leave them unoriented for now. Now, the mixed graph $G_r$ is constructed by taking two copies of $H$, $H_1$ and $H_2$ and identifying the root $a$ of $H_1$ and $H_2$. The vertices of $H_1$ retain their labels from $H$. We can easily see that $G_r$ is strongly connected with radius $r$ and $a$ is the unique central vertex. Figure~\ref{fig-mmg} shows the mixed multigraph $G_3$.  

It can be verified that in any strong orientation $\ora{G_r}$ that minimizes the radius, $a$ is the unique central vertex. Hence, there will be an optimum orientation in which the orientation of $H_1$ is symmetric with the orientation of $H_2$. We have already seen that edges from a middle child to its parent are the only edges in $G_r$ whose orientation is not forced in a strong orientation. Hence, in any strong orientation $\ora{G_r}$ of $G_r$ with an induced orientation $\ora{H_1}$ of $H_1$, we have the following bounds for every internal vertex $p_i$ of $T_r$ at Level~$i$ with $l_{i+1}$ and $r_{i+1}$ as the left and right children, respectively, for $0 \leq i \leq r-1$. 
\begin{align}
d_{\ora{H_1}}(l_{i+1},p_{i}) & \in \{ 2(r-i), 4(r-i)-1 \} \nonumber \\
d_{\ora{H_1}}(p_{i},r_{i+1}) & \in \{ 2(r-i), 4(r-i)-1 \}  \nonumber \\
d_{\ora{H_1}}(p_{i},l_{i+1}) & = 1 \nonumber \\
d_{\ora{H_1}}(r_{i+1},p_{i}) & = 1 
\end{align} 

From the first line of Equation~$5$, it follows that,  $d_{\ora{H_1}}(f,b) \geq  \Sigma_{i=1}^{r-1} 2(r-i) =r^2-r$. Similarly, from the second line of Equation~$5$, it follows that, $d_{\ora{H_1}}(d,g) \geq  \Sigma_{i=1}^{r-1} 2(r-i) =r^2-r$. There are only two possible orientations for the edge $ac$. If the edge $ac$ in $H_1$ is oriented from $a$ to $c$, then $d_{\ora{H_1}}(b,a) = 4r-1$. Hence $d_{\ora{H_1}}(f,a) =  d_{\ora{H_1}}(f,b) + d_{\ora{H_1}}(b,a) \geq (r^2-r) + (4r-1) = r^2+3r-1$. Similarly, if the edge $ac$ in $\ora{H_1}$ is oriented from $c$ to $a$, we can argue that  $d_{\ora{H_1}}(a,g) = d_{\ora{H_1}}(a,d) + d_{\ora{H_1}}(d,g) \geq (4r-1)+(r^2-r) = r^2+3r-1$. From this, it follows that, $G_r$ has oriented radius at least $r^2+3r-1$. 

Now, we show that there is an orientation $\ora{G_r}$ with radius $r^2+3r-1$. Recall that all the edges except the ones between the middle children and their parents are forced in any strong orientation $\ora{H_1}$ of $H_1$. First, we orient all the edges whose orientation is forced. Now, the only edges that remain to be oriented are those which go from a parent to its middle child. Among them, the edge $ac$ of $H_1$ is oriented from $a$ to $c$. Every edge from a vertex $y$ to its middle child is oriented away from $y$ if the edge from $y$ to its parent is oriented towards $y$ and vice-versa. This completes the orientation of $\ora{G_r}$.  

Let $v_r$ be a vertex with maximum in-distance to $a$. It can be verified that $v_r$ is a left or right child in Level~$r$ of $T_r$. Let $P$ be a shortest path from $v_r$ to $a$ in $\ora{H_1}$ and let $v_r, v_{r-1}, \cdots, v_0=a$ be a subsequence of $P$ such that for $1 \leq i \leq r$, $v_{i-1}$ is the parent of $v_i$ in $T_r$. Let $d_i$ denote the distance from $v_{i}$ to $v_{i-1}$. By Equation~$5$, $d_i \in \{1, 2(r-i+1), 4(r-i+1)-1\}$. We can see that $d_{\ora{H_1}}(v_r,a)= d_r + d_{r-1} + \cdots + d_1$. We have already seen that $d_i \leq 4(r-i+1)-1$. 
Further, if $d_i = 4(r-i+1)-1$, then the edge between $v_{i-1}$ and its middle child is oriented away from $v_{i-1}$. Hence, the edge between $v_{i-1}$ and its parent $v_{i-2}$ is oriented away from $v_{i-1}$ and hence $d_{i-1} = 1$, for $2 \leq i \leq r $. Hence, for $2 \leq j \leq r-1$, it can be verified using induction on $(r-j)$ that, if $\Sigma_{i=j}^r d_i > \Sigma_{i=j}^r  2(r-i+1)$, then $\Sigma_{i=j-1}^{r} d_i \leq \Sigma_{i=j-1}^{r} 2(r-i+1)$. Choosing $j =2$, we see that, either $d_2 + \cdots + d_r \leq  \Sigma_{i=2}^{r} 2(r-i+1) = r^2-r$ or $d_1 + d_2 + \cdots + d_r \leq \Sigma_{i=1}^{r} 2(r-i+1) = r^2+r \leq r^2 + 3r -1$. Since $d_1 \leq 4r - 1$,  in the former case too, $d_1 + d_2 + \cdots + d_r \leq r^2 + 3r - 1$. Thus, $d_{\ora{H_1}}(v_r,a) \leq r^2+3r-1$. Hence, $\forall v \in V(H_1)$, $ d_{\ora{H_1}}(v,a) \leq r^2+3r-1$. Similarly, we can prove that $\forall v \in V(H_1)$, $d_{\ora{H_1}}(a,v) \leq r^2+3r-1$. Since $H_1$ and $H_2$ are isomorphic, there exists a strong orientation $\ora{G_r}$ of $G_r$ with radius at most $r^2+3r-1$. Since we have already shown that $G_r$ has oriented radius at least $r^2+3r-1$, the oriented radius of $G_r$ is $r^2+3r-1$.  

Now the main theorem of the section follows.               

\begin{theorem}\label{lbmmg}
$f(r) \geq r^2+3r-1$. 
\end{theorem} 

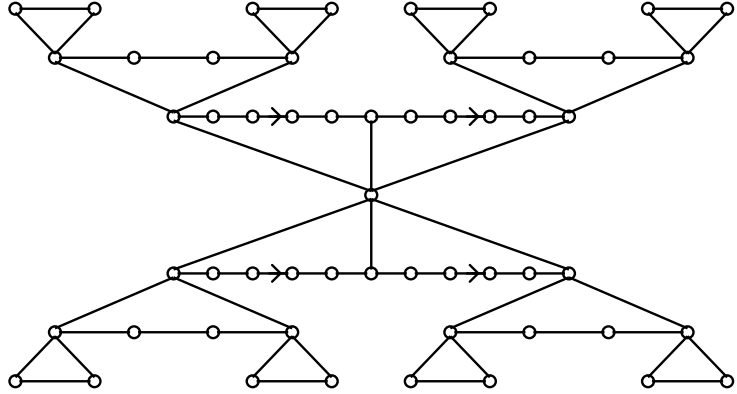
\begin{figure}[ht]
    
    \centering
   
\begin{tikzpicture}[scale=0.26]
\begin{scope}[every node/.style={sloped,allow upside down}]
\draw (0,2) circle (0.3cm); 

\draw (0,6) circle (0.3cm); 

\draw (2,6) circle (0.3cm);
\draw (4,6) circle (0.3cm);
\draw (6,6) circle (0.3cm);
\draw (8,6) circle (0.3cm);
\draw (10,6) circle (0.3cm); 

\draw (-2,6) circle (0.3cm);
\draw (-4,6) circle (0.3cm);
\draw (-6,6) circle (0.3cm);
\draw (-8,6) circle (0.3cm);
\draw (-10,6) circle (0.3cm); 


\draw (12,9) circle (0.3cm);
\draw (16,9) circle (0.3cm);
\draw (8,9) circle (0.3cm);
\draw (4,9) circle (0.3cm);


\draw (-12,9) circle (0.3cm);
\draw (-16,9) circle (0.3cm);
\draw (-8,9) circle (0.3cm);
\draw (-4,9) circle (0.3cm);


\draw (18,11.5) circle (0.3cm);
\draw (14,11.5) circle (0.3cm);


\draw (-2,11.5) circle (0.3cm);
\draw (-6,11.5) circle (0.3cm);


\draw (2,11.5) circle (0.3cm);
\draw (6,11.5) circle (0.3cm);


\draw (-18,11.5) circle (0.3cm);
\draw (-14,11.5) circle (0.3cm);

\draw (-17.80,11.5)--(-14.20,11.5);

\draw (-5.80,11.5)--(-2.20,11.5);

\draw (2.20,11.5)--(5.80,11.5);

\draw (14.20,11.5)--(17.80,11.5);

\draw (-15.80,9)--(-12.20,9);
\draw (-11.80,9)--(-8.20,9);
\draw (-7.80,9)--(-4.20,9);

\draw (4.20,9)--(7.80,9);
\draw (8.20,9)--(11.80,9);
\draw (12.20,9)--(15.80,9);

\draw (-9.80,6)--(-8.20,6);
\draw (-7.80,6)--(-6.20,6);
\draw (-5.80,6)--(-4.20,6);
\draw (-5.20,6)--node {\midarrowa}(-4.20,6);
\draw (-3.80,6)--(-2.20,6);
\draw (-1.80,6)--(-0.20,6);
\draw (0.20,6)--(1.80,6);
\draw (2.20,6)--(3.80,6);
\draw (4.20,6)--(5.80,6);
\draw (4.80,6)--node {\midarrowa}(5.80,6);
\draw (6.20,6)--(7.80,6);
\draw (8.20,6)--(9.80,6);

\draw (0,2.2)--(0,5.8);



\draw (-16,9.2)--(-18,11.3);
\draw (-4,9.2)--(-6,11.3);
\draw (-16,9.2)--(-14,11.3);
\draw (-4,9.2)--(-2,11.3);

\draw (16,9.2)--(18,11.3);
\draw (4,9.2)--(6,11.3);
\draw (16,9.2)--(14,11.3);
\draw (4,9.2)--(2,11.3);

\draw (10,6.2)--(16,8.8);
\draw (10,6.2)--(4,8.8);

\draw (-10,6.2)--(-16,8.8);
\draw (-10,6.2)--(-4,8.8);

\draw (0,2.2)--(10,5.8);
\draw (0,2.2)--(-10,5.8);


\draw (0,-2) circle (0.3cm);

\draw (2,-2) circle (0.3cm);
\draw (4,-2) circle (0.3cm);
\draw (6,-2) circle (0.3cm);
\draw (8,-2) circle (0.3cm);
\draw (10,-2) circle (0.3cm);

\draw (-2,-2) circle (0.3cm);
\draw (-4,-2) circle (0.3cm);
\draw (-6,-2) circle (0.3cm);
\draw (-8,-2) circle (0.3cm);
\draw (-10,-2) circle (0.3cm);


\draw (12,-5) circle (0.3cm);
\draw (16,-5) circle (0.3cm);
\draw (8,-5) circle (0.3cm);
\draw (4,-5) circle (0.3cm);


\draw (-12,-5) circle (0.3cm);
\draw (-16,-5) circle (0.3cm);
\draw (-8,-5) circle (0.3cm);
\draw (-4,-5) circle (0.3cm);


\draw (18,-7.5) circle (0.3cm);
\draw (14,-7.5) circle (0.3cm);


\draw (-2,-7.5) circle (0.3cm);
\draw (-6,-7.5) circle (0.3cm);


\draw (2,-7.5) circle (0.3cm);
\draw (6,-7.5) circle (0.3cm);


\draw (-18,-7.5) circle (0.3cm);
\draw (-14,-7.5) circle (0.3cm);

\draw (-17.80,-7.5)--(-14.20,-7.5);

\draw (-5.80,-7.5)--(-2.20,-7.5);

\draw (2.20,-7.5)--(5.80,-7.5);

\draw (14.20,-7.5)--(17.80,-7.5);

\draw (-15.80,-5)--(-12.20,-5);
\draw (-11.80,-5)--(-8.20,-5);
\draw (-7.80,-5)--(-4.20,-5);

\draw (4.20,-5)--(7.80,-5);
\draw (8.20,-5)--(11.80,-5);
\draw (12.20,-5)--(15.80,-5);

\draw (-9.80,-2)--(-8.20,-2);
\draw (-7.80,-2)--(-6.20,-2);
\draw (-5.80,-2)--(-4.20,-2);
\draw (-5.20,-2)--node {\midarrowa}(-4.20,-2);
\draw (-3.80,-2)--(-2.20,-2);
\draw (-1.80,-2)--(-0.20,-2);
\draw (0.20,-2)--(1.80,-2);
\draw (2.20,-2)--(3.80,-2);
\draw (4.20,-2)--(5.80,-2);
\draw (4.80,-2)--node {\midarrowa}(5.80,-2);
\draw (6.20,-2)--(7.80,-2);
\draw (8.20,-2)--(9.80,-2);

\draw (0,1.8)--(0,-1.8);



\draw (-16,-5.2)--(-18,-7.3);
\draw (-4,-5.2)--(-6,-7.3);
\draw (-16,-5.2)--(-14,-7.3);
\draw (-4,-5.2)--(-2,-7.3);

\draw (16,-5.2)--(18,-7.3);
\draw (4,-5.2)--(6,-7.3);
\draw (16,-5.2)--(14,-7.3);
\draw (4,-5.2)--(2,-7.3);

\draw (10,-2.2)--(16,-4.8);
\draw (10,-2.2)--(4,-4.8);

\draw (-10,-2.2)--(-16,-4.8);
\draw (-10,-2.2)--(-4,-4.8);

\draw (0,1.8)--(10,-1.8);
\draw (0,1.8)--(-10,-1.8);

\end{scope}
\end{tikzpicture}

    \caption{A Mixed multigraph of radius $3$ and oriented radius $17$.}
    \label{fig:fig3}
\end{figure}

Notice that, the same bound could be obtained by modifying only the first level in the example of Chv{\'a}tal and Thomassen (See Figure~\ref{fig:fig3}). But still, we chose to analyze $G_r$, because our algorithm will orient $G_r$ with an oriented radius of $1.5r^2+O(r)$ only. That is, $G_r$ demonstrates the tightness of the analysis of our algorithm. Hence, there is scope for a better algorithm and we conjecture that $f(r)=r^2+O(r)$.  

\begin{conjecture}\label{c_1}
$f(r) = r^2+O(r)$. 
\end{conjecture}


\section*{Appendix}

\odds*

\begin{proof}
If $n=2$, then the partition $(\{1\},\{2\})$ suffices. Hence we assume $n \neq 2$. 

Any four consecutive odd numbers $a_1,a_2,a_3,a_4$ can be split into two parts $\{a_1,a_4\}$ and $\{a_2,a_3\}$ with equal sum. We partition every block of four consecutive odd numbers in $S$ this way, starting from the largest and going backwards, till we are left with $0,1,6$ or $3$ smallest odd numbers. So far the two parts have equal sum. The partition    
 $(\{\phi\},\{\phi\})$, $(\{\phi\},\{1\})$, $(\{1,3,5,9\},\{7,11\})$, $(\{1,3\},\{5\})$, respectively for each case above can be appended to get the required partition. 

It can be verified that the partition can be generated in polynomial time.
\end{proof}

Chung, Garey and Tarjan proved the following.
\begin{lemma}\cite[Lemma~$5$]{chung1985strongly} 
Let $G$ be a strongly connected bridgeless mixed multigraph of radius $r$ with a central vertex $u$. Then any edge incident to $u$ is on a cycle of length at most $2r+1$. 
\end{lemma}

We generalize this result by extending it to every edge in $G$ (Lemma~\ref{tarlemma}). The proof given in \cite{chung1985strongly} can be used to show that every edge $pq$ of $G$ is on a cycle of length at most $2k+1$, where $k$ is the eccentricity of $p$. In fact further results in \cite{chung1985strongly} make use of this version, even though it is not explicitly stated so. But, we argue in Lemma~\ref{tarlemma} that any edge $pq$ of $G$ is on a cycle of length at most $2r+1$ itself. The proof of Lemma~\ref{tarlemma} given below is an extension of the proof given by Chung, Garey and Tarjan \cite{chung1985strongly}.
\goldbach*

\begin{proof}
Let $G$ be a strongly connected bridgeless mixed multigraph of radius $r$ with a central vertex $u$. First we consider the case when $\ora{pq}$ is a directed edge oriented from $p$ to $q$. Since $G$ is of radius $r$, there exists a $u-p$ path of length at most $r$ and a $q-u$ path of length at most $r$. Hence, there is a walk and therefore a path of length at most $2r$ from $q$ to $p$. Together with $\ora{pq}$, it forms a cycle of length at most $2r+1$ containing $\ora{pq}$. 

Now, let $pq$ be an undirected edge. In order to show that an edge $pq$ lies in a cycle of length at most $2r+1$, it is enough to show that it lies in a closed walk of length at most $2r+1$ which passes through $pq$ exactly once. We define two subsets $X$ and $Y$ of $V(G)$ as follows. A vertex $v \in V(G)$ is in $X$ if every shortest path from $u$ to $v$ contains $pq$. A vertex $v \in V(G)$ is in $Y$ if every shortest path from $v$ to $u$ contains $pq$. First, we handle the case when $X \neq Y$. Assume, there exists a vertex $w \in (X \setminus Y)$. Since $w \in X$, all the shortest paths from $u$ to $w$ contains $pq$. Let $P_w$ be one such path. Since $w \notin Y$, at least one of the shortest paths from $w$ to $u$ do not pass through $pq$. Let $P_u$ be one such path. Now, $P_w \cup P_u$ is a closed walk of length at most $2r$ containing the edge $pq$ exactly once. A similar argument could also be made, if there exists a vertex $w \in (Y \setminus X)$. Henceforth, we assume, $X=Y$. 

Let $\overline{X} = V(G) \setminus X$ and $\overline{Y} = V(G) \setminus Y$. Notice that, since $X=Y$, $\overline{X}=\overline{Y}$. Further notice that $\overline{X} \ne \phi$ as $u \in \overline{X}$. Now, assume vertices $p, q \in \overline{X}$. Since $\overline{X}=\overline{Y}$, there is a shortest path from $u$ to $p$ of length at most $r$ and a shortest path from $q$ to $u$ of length at most $r$, both of which  do not pass through $pq$. Hence, it is easy to see that there is a closed walk of length at most $2r+1$ containing $pq$ exactly once. Next we show that $p$ and $q$ cannot both be in $X$. If $q \in X$, by the definition of $X$, every shortest path from $u$ to $q$ passes through $pq$. Any of those paths on removing the edge $pq$ will give a shortest path from $u$ to $p$ not containing $pq$. Hence, $p \notin X$. Similarly, if $p \in X$ then $q \notin X$. 

Henceforth, we can assume that $pq$ crosses the cut $(X,\overline{X})$.
Since, $G$ does not contain any bridges, there is an edge $ab$, other than $pq$, crossing the cut $(X, \overline{X})$ with $a \in X$ and $b \in \overline{X}$. First, we consider the case when the edge $ab$ is undirected or directed from $a$ to $b$. Since $a \in X$, all the shortest paths from $u$ to $a$ contains the edge $pq$. Let $P_a$ be one such path. Since $b \in \overline{Y}$, there exists at least one shortest path $P_u$ from $b$ to $u$ which does not contain $pq$. Hence, $P_a \cup P_u \cup \{ab\}$ is a closed walk of length at most $2r+1$ containing the edge $pq$ exactly once. Finally, we consider the case when $ab$ is oriented from $b$ to $a$. Since $a \in Y$, every shortest path from $a$ to $u$ passes through the edge $pq$. Since $b \notin X$, there exists at least one shortest path from $u$ to $b$ which does not contain the edge $pq$. Hence, in this case also, there exists a closed walk of length at most $2r+1$ containing the edge $pq$ exactly once.                         
\end{proof}

\oddevens*

\begin{proof}
When $n = 2$, $S_1=\{1, 3\}$ which can be partitioned into two sets $X$ and $Y$ such that $\lvert \Sigma_{a \in X} a - \Sigma_{b \in Y} b \rvert = 2$. When $n \neq 2$, by Lemma~\ref{oddlemma}, $S_1$ can be partitioned into two sets $X$ and $Y$ such that $\lvert \Sigma_{a \in X} a - \Sigma_{b \in Y} b \rvert \leq 1$. 
\begin{itemize}

\item
If $k$ is even, $S_2$ can be partitioned into two multisets 
$P$ and $Q$ with $\frac{k}{2}$ elements each. The sum of elements in both $P$ and $Q$ are the same. 
\item
If $k$ is odd, let $S_1'=S_1 \setminus \{2n-1\}$ and $S_2'$ be the multiset $S_2 \cup \{2n-1\}$.  As done before, $S_1'$ can be partitioned into two sets $X$ and $Y$ such that $ \Sigma_{b \in Y} b \leq \Sigma_{a \in X} a \leq (\Sigma_{b \in Y} b) + 2$. Further, since the value of the elements in the multiset $S_2$ is either $2n-1$ or $2n$, $S_2'$ can be partitioned into two multisets $P$ and $Q$ such that $\Sigma_{b \in Q} b \leq \Sigma_{a \in P} a \leq (\Sigma_{b \in Q} b) + 1$.
\end{itemize}
Now, $A=X \cup Q$ and $B = Y \cup P$ is the required partition of $S$.

It can be verified that the partition can be generated in polynomial time. 
\end{proof} 

\end{document}